\documentclass[12pt]{amsart}

\theoremstyle{plain}

\newtheorem{theorem}{Theorem}[section]

\newtheorem{lemma}[theorem]{Lemma}
\newtheorem{proposition}[theorem]{Proposition}
\newtheorem{corollary}[theorem]{Corollary}
\newtheorem{observation}[theorem]{Observation}

\theoremstyle{definition}

\newtheorem{definition}[theorem]{Definition}
\newtheorem{remark}[theorem]{Remark}
\newtheorem{example}[theorem]{Example}
\newtheorem{notation}[theorem]{Notation}

\newcommand{\ep}{\varepsilon}

\newcommand{\cc}{\subset\subset}
\newcommand{\sipky}{\nearrow\!\!\!\nearrow}
\newcommand{\sipka}{\nearrow}
\newcommand{\dist}{\mathrm{dist}}

\newcommand{\R}{\mathbb{R}}
\newcommand{\N}{\mathbb{N}}

\renewcommand{\epsilon}{\varepsilon}
\renewcommand{\phi}{\varphi}
\renewcommand{\tilde}{\widetilde}

\usepackage{amssymb}
\usepackage{amsmath}
\usepackage{amsthm}

\def \cal{\mathcal}

\begin{document}

\title{\large On extensions of d.c.\ functions\\
and convex functions}

\author{Libor Vesel\'y}
\address{Dipartimento di Matematica\\
Universit\`a degli Studi\\
Via C.~Saldini 50\\
20133 Milano\\
Italy}

\author{Lud\v ek Zaj\'\i\v cek}
\address{Charles University\\
Faculty of Mathematics and Physics\\
Sokolovsk\'a 83\\
186 75 Praha 8\\
Czech Republic}

\email{vesely@mat.unimi.it}
\email{zajicek@karlin.mff.cuni.cz}

 \subjclass[2000]{Primary 52A41; Secondary 26B25, 46B99}

 \keywords{d.c.\ function, d.c.\ mapping, delta-convex mapping, convex function, extension, Banach space}

 \thanks{}

\begin{abstract}
We show how our recent results on compositions of d.c.\ functions (and
mappings) imply positive results on extensions of d.c.\ functions (and
mappings). Examples answering two natural relevant questions are
presented. Two further theorems, concerning extendability of continuous
convex functions from a closed subspace of a normed linear space,
complement recent results of J.~Borwein, V.~Montesinos and J.~Vanderwerff.
\end{abstract}

\maketitle


\markboth{L.~Vesel\'y and L.~Zaj\'{\i}\v{c}ek}{On extensions of d.c.\ functions and convex functions}

\section*{Introduction}

Let $C$ be a nonempty convex set in a (real) normed linear space $X$. A
function $f\colon C\to\R$ is called {\em d.c.}\ (or ``delta-convex'')
if it can be represented as a difference of two continuous convex
functions on $C$. An extension of this notion, the notion of a {\em d.c.\ mapping}
$F\colon C\to Y$ (see Definition~\ref{D:dc}) where $Y$ is a normed linear space,
was introduced in \cite{VeZa1} and studied in \cite{VeZa1},
\cite{DuVeZa}, \cite{VeZa2} and some other papers by the authors.

The present paper concerns the following natural questions.
\begin{enumerate}
\item[(Q1)] When is it possible to extend a d.c.\ function (or a d.c.\ mapping) on $C$
to a d.c.\ function (or a d.c.\ mapping) on the whole $X$?
\item[(Q2)] When is it possible to extend a continuous convex function
on a closed subspace $Y$
of $X$ to a continuous convex function on $X$?
\end{enumerate}

In Section 2, we show how results of \cite{VeZa2} on compositions of
d.c.\ functions and mappings imply positive results concerning (Q1).
For instance, Corollary~\ref{elpecka}(a) reads as follows.

{\em
Let $X$ be a (subspace of some) $L_p(\mu)$ space with $1<p\le2$. Let
$C\subset X$ be a convex set with a nonempty interior. Then each continuous
convex function  $f$ on $C$, which is Lipschitz on every bounded subset of
$\mathrm{int}\,C$, admits a d.c.\ extension to the whole $X$.}

\noindent
(Note that only the
case of $C$ unbounded is interesting; cf.~Lemma~\ref{F:1}(c).)
The needed results from \cite{VeZa2}, together with some definitions and
auxiliary facts, are presented in Section~1 (Preliminaries).

Section 3 contains two counterexamples.
The first one (Example~\ref{pasovec})
shows that, in the above mentioned
Corollary~\ref{elpecka}(a), we cannot conclude that $f$ admits a
continuous {\em convex} extension (even for $X= \R^2$).
The second counterexample (Example~\ref{koulovec}) shows that,
in the above mentioned Corollary~\ref{elpecka}(a),
it is not possible
to relax the assumption that $f$ is Lipschitz on bounded sets by assuming that $f$ is only
locally Lipschitz on $C$.

In the last Section~4, we consider the question (Q2)
of extendability of
continuous convex functions from a closed subspace $Y$ to the whole $X$.
The authors of~\cite{BMV} obtained a necessary and
sufficient condition on $Y$ in terms of nets in $Y^*$ and, using
Rosenthal's extension theorem, they proved the following interesting
corollary (\cite[Corollary~4.10]{BMV}).

{\em If $X$ is a Banach space and $X/Y$ is separable, then each continuous convex function on
$Y$ admits a continuous convex extension to $X$.}

\noindent
 Using methods from \cite{H} and \cite{VeZa2}, we give a
 necessary and sufficient condition on $Y$ of a different type in Theorem~\ref{univ}.
As an application, we present an
elementary alternative proof of the above mentioned
  \cite[Corollary~4.10]{BMV}, which works also for noncomplete $X$.


\section{Preliminaries}\label{S:prelim}

We consider only normed linear spaces over the
reals $\R$. For a normed linear space $X$ we use the following
fairly standard notations:
$B_X$ denotes the closed unit ball; $U(c,r)$ is the open ball centered in $c$ with radius $r$;
$[x,y]$ is the closed segment $\mathrm{conv}\{x,y\}$ (the meaning of the symbols
$(x,y)$ and $(x,y]=[y,x)$ is clear). By definition, the distance of a
set from the empty set $\emptyset$ is $\infty$, and the restriction of a
mapping to $\emptyset$ has all properties like continuity, Lipschitz
property, boundedness and so on.

We will frequently use also the following less standard notation.

\begin{notation}
Let $A,B,A_n,B_n$ ($n\in\N$) be subsets of a normed linear space $X$. We
shall write:
\begin{itemize}
\item $A\subset\subset B$ whenever there exists $\epsilon>0$
such that $A+\epsilon B_X\subset B$;
\item $A_n\nearrow A$ whenever $A_n\subset A_{n+1}$ for each
$n\in\N$, and $\bigcup_{n\in\N}A_n=A$;
\item $A_n\nearrow\!\!\!\nearrow A$ whenever $A_n\subset\subset A_{n+1}$ for each
$n\in\N$, and $\bigcup_{n\in\N}A_n=A$.
\end{itemize}
\end{notation}

We shall use the following simple facts about convex sets and functions.

\begin{lemma}[{\cite[Lemma~2.3]{VeZa2}}]\label{L:dn}
Let $C\subset X$ be nonempty, open and convex. Let $\{C_n\}$ be a sequence of
convex sets with nonempty interiors, such that $C_n\sipka C$.
Then there exists a sequence $\{D_n\}$ of nonempty,
bounded, open,
convex sets such that $D_n\sipky C$, and $D_n\subset\subset C_n$ for each $n$.
\end{lemma}

\begin{lemma}[{\cite[Fact~1.6]{VeZa2}}]\label{F:1}
Let $C\subset X$ be a nonempty convex set, $f\colon C\to\R$ be a
convex function.
\begin{enumerate}
\item[(a)] If $C$ is open and bounded and $f$ is continuous,
           then $f$ is bounded below on $C$.
\item[(b)] If $f$ is bounded on $C$ then $f$ is Lipschitz on each $D\subset\subset C$.
\item[(c)] If $f$ is Lipschitz then it admits a Lipschitz convex
extension to $X$.
\end{enumerate}
\end{lemma}

\begin{lemma}\label{K}
Let $f$ be a continuous convex function on an open convex subset $C$ of
a normed linear space $X$. Then there exists a sequence $\{D_n\}$ of
nonempty bounded open convex sets such that $D_n\sipka C$ and $f$ is
Lipschitz (and hence bounded) on each $D_n$.
\end{lemma}

\begin{proof}
Fix $x_0\in C$ and consider the nonempty open convex sets
$C_n:=\{x\in C: f(x)<f(x_0)+n\}$. By Lemma~\ref{L:dn}, there exist
nonempty bounded open convex sets $D_n$ such that $D_n\cc C_n$ and
$D_n\sipky C$. Using Lemma~\ref{F:1}(a), it is easy to see that $f$ is
bounded on each $D_{n+1}$. Hence, by Lemma~\ref{F:1}(b), $f$ is Lipschitz
on each $D_n$.
\end{proof}

Let us recall the following easy known fact (see, e.g., \cite[Theorem~1.25]{valentine}):
if $A,B$ are convex sets in
a vector space then
\begin{equation}\label{Val}
\mathrm{conv}(A\cup B)=\bigcup_{0\le t\le1}[(1-t)A+tB]=\bigcup_{a\in A,\, b\in B}[a,b]\,.
\end{equation}

\begin{lemma}\label{conv}
Let $Y$ be a closed subspace of a normed linear space $X$,
$C\subset Y$ and $A\subset X$ convex sets.
\begin{enumerate}
\item[(a)] $\mathrm{conv}(A\cup C)\cap Y=\mathrm{conv}[(Y\cap A)\cup C]$.
\item[(b)] If $\mathrm{int}\,A\ne\emptyset$ and $A$ is dense in $X$,
then $A=X$.
\item[(c)] If $C$ is open in $Y$, $A$ is open in $X$ and
$A\cap C\ne\emptyset$, then $\mathrm{conv}(A\cup C)$ is open.
\end{enumerate}
\end{lemma}

\begin{proof}
(a) The inclusion ``$\supset$'' is obvious. To prove the other
inclusion, consider an arbitrary $y\in Y\cap\mathrm{conv}(A\cup C)$.
Then $y\in[a,c]$ for some $a\in A$, $c\in C$. If $y\ne c$ then
necessarily $a\in Y$ (since $y,c\in Y$) and hence
$y\in \mathrm{conv}[(Y\cap A)\cup C]$; and the last formula is trivial
for $y=c$.\\
(b) follows, e.g., from the well-known fact that
$\mathrm{int}(\overline{A})=\mathrm{int}\,A$ whenever
$\mathrm{int}\,A$ is nonempty.\\
(c) Fix an arbitrary $a_0\in A\cap C$. For each $x\in C$, there
obviously exists $y\in C\setminus\{a_0\}$ such that $x\in(y,a_0]$;
consequently, there exists
$t\in(0,1]$ with $x\in(1-t)C+tA$. Now we are done, since
$$\mathrm{conv}(A\cup C)=
C\cup\bigcup_{0<t\le1}[(1-t)C+tA]=
\bigcup_{0<t\le1}[(1-t)C+tA]$$
and the members of the last union are open.
\end{proof}

In the rest of this section, we collect some facts about d.c.\ functions
and mappings, which we will need in the next sections.

Let $C$ be a convex set in a normed linear space $X$. Recall that a
function $f\colon C\to\R$ is {\em d.c.}\ (or ``delta-convex'')
if it can be represented as a difference of two continuous convex
functions on $C$. The following generalization to the case of vector-valued
mappings on $C$ was studied in \cite{VeZa1} for open $C$, and in
\cite{VeZa2} for a general (convex) $C$.

\begin{definition}\label{D:dc}
Let $X,Y$ be normed linear spaces, $C\subset X$ be a convex set, and
$F\colon C\to Y$ be a continuous mapping. We say that $F$ is {\em d.c.}\
(or ``delta-convex'') if there exists a continuous (necessarily convex)
function $f\colon C\to\R$ such that $y^*\circ F+f$ is convex on $C$
whenever $y^*\in Y^*$, $\|y^*\|\le1$. In this case we say that $f$
controls $F$, or that $f$ is a {\em control function} for $F$.
\end{definition}

\begin{remark}\label{R:dc}
It is easy to see (cf.~\cite{VeZa1}) that:
\begin{enumerate}
\item[(a)] a mapping
$F=(F_1,\ldots,F_m)\colon C\to\R^m$ is d.c.\ if and only if each of its
components $F_k$ is a d.c.\ function;
\item[(b)] the notion of delta-convexity does not depend on the choice
of equivalent norms on $X$ and $Y$.
\end{enumerate}
\end{remark}

\begin{lemma}[{\cite[Lemma~5.1]{VeZa2}}]\label{ndc}
Let $X,Y$ be normed linear spaces, let $A\subset X$ be an open convex
set with $0\in A$, and let $F\colon A\to Y$ be a mapping. Suppose there exist
$\lambda\in(0,1)$ and a sequence of balls $B(x_n,\delta_n)\subset A$ such
that $\{x_n\}\subset\lambda A$, $\delta_n\to0$ and $F$ is unbounded on
each $B(x_n,\delta_n)$. Then $F$ is not d.c.\ on $A$.
\end{lemma}

The following result was proved in \cite[Theorem~18(i)]{DuVeZa} for
$X,Y$ Banach spaces, but the proof therein works for normed linear
spaces as well.

\begin{proposition}\label{P:lip}
Let $X,Y$ be normed linear spaces, $C\subset X$ be a bounded open convex
set, and $F\colon C\to Y$ be a d.c.\ mapping with a Lipschitz control
function. Then $F$ is Lipschitz.
\end{proposition}

\begin{lemma}[{\cite[Lemma~2.1]{VeZa2}}]\label{L:hartman}
Let $X,Y$ be normed linear spaces, $C\subset X$ a nonempty
convex set, and  $F\colon C\to Y$ a mapping. Let $\emptyset \neq D_n \subset C$
($n\in\N$)
be convex sets such that $D_n\sipka C$ and, for each $n$,
$\mathrm{dist}(D_n,C\setminus D_{n+1})>0$,
$D_n$ is relatively open in $C$,
and
$F|_{D_n}$ is d.c.\ with a
control function  $\gamma_n\colon D_n\to \R$
which is either bounded or Lipschitz. Then $F$ is d.c.\ on $C$.
\end{lemma}

An important ingredient of the present paper is application of the
following two results on compositions of d.c.\ mappings.

\begin{proposition}[\cite{VeZa1},\cite{VeZa2}]\label{P:veza}
Let $X,Y,Z$ be normed linear spaces, $A\subset X$ and $B\subset Y$
convex sets, and $F\colon A\to B$ and $G\colon B\to Z$ d.c.\ mappings.
If $G$ is Lipschitz and has a Lipschitz control function,
then $G\circ F$ is d.c.\ on $A$.
\end{proposition}

\begin{lemma}[{\cite[Lemma~3.2(ii)]{VeZa2}}]\label{P:main}
Let $U,V,W$ be normed linear spaces, let $A\subset U$ be an open convex
set
and $B\subset V$ a convex set, and let $\Phi\colon A\to B$
 and $\Psi\colon B\to W$ be mappings.
Suppose that $\Phi$ is d.c.\ and
there exist sequences of convex sets $A_n \subset A$, $B_n \subset B$
such that
$\mathrm{int}\,A_n \neq \emptyset$,
$A_n\sipka A$,
$\Phi(A_n) \subset B_n$, and $\Psi|_{B_n}$ is Lipschitz and d.c.\ with a
Lipschitz control function.
 Then $\Psi\circ \Phi$ is d.c.\ on $A$.
 \end{lemma}

Let us recall that a normed linear space $X$ is said to have {\em
modulus of convexity of power type 2} if there exists $a>0$ such that $\delta_X(\epsilon)\ge
a\epsilon^2$ for each $\epsilon\in(0,2]$ (where $\delta_X$ denotes the
classical modulus of convexity of $X$; see e.g.\ \cite[p.409]{BeLin} for the
definition).

\begin{proposition}[{\cite[Corollary~3.9(a)]{VeZa2}}]\label{bilinear}
Let $Y,V,X,Z$ be normed linear spa\-ces and let both
$Y$ and $V$ admit renormings with modulus of
convexity of power type 2. Let $B\colon Y\times V\to Z$ be a
continuous bilinear mapping,
 $C\subset X$  an open convex set, and let
$F\colon C\to Y$ and $G\colon C\to V$ be d.c.\ mappings. Then the
mapping $B\circ(F,G)\colon x\mapsto B(F(x),G(x))$ is d.c.\ on $C$.
\end{proposition}


\section{Extensions of d.c.\ mappings}

Let $X,Y$ be normed linear spaces, $C\subset X$ be a convex set, and
$F\colon C\to Y$ be a d.c.\ mapping.
In the present section, we are interested in existence of a d.c.\
extension of $F$ to the whole $X$ or at least to the closure of $C$. Let
us start with a simple observation.

\begin{observation}\label{O:exten}
Let $X,Y,C,F$ be as in the beginning of this section, and
$f\colon C\to\R$ a control function of $F$.
\begin{enumerate}
\item[(a)] If $Y$ is finite-dimensional, and both $F,f$ are Lipschitz on
$C$, then $F$ admits a d.c.\ extension to $X$.
\item[(b)] If both $F,f$ admit continuous extensions $\tilde{F},\tilde{f}$
to a convex set $D$ such that $C\subset D\subset\overline{C}$,
then $\tilde F$ is d.c.\ with the control function $\tilde f$.
\end{enumerate}
\end{observation}

\begin{proof}
 By Remark~\ref{R:dc}, it suffices to prove {\it(a)} for $Y=\R$. In this
case, $F=(F+f)-f$ is a difference of two Lipschitz convex functions on
$C$. By Lemma~\ref{F:1}(c), $F$ can be extended to a difference of two
Lipschitz convex functions on $X$. The assertion
{\it(b)} follows by a simple limit argument.
\end{proof}

\begin{proposition}\label{L:uzaver}
Let $X$ be a normed linear space, $Y$ a Banach space, $C\subset X$
a convex set with a nonempty interior and $F\colon C\to Y$ a d.c.\ mapping.
Suppose there exists a nondecreasing sequence $\{A_n\}$ of open convex
sets in $X$ such that $\overline{C}\subset \bigcup A_n$ and, for each
$n$, $F|_{(\mathrm{int}\,C)\cap A_n}$ has a Lipschitz control function.
Then $F$ admits a
d.c.\ extension to $\overline{C}$.
\end{proposition}

\begin{proof}
Since $A_n\sipka A:=\bigcup_k A_k$,
Lemma~\ref{L:dn} allows us to suppose that the sets $A_n$ are also
bounded and satisfy $A_n\sipky A$. By Proposition~\ref{P:lip},
$F|_{(\mathrm{int}\,C)\cap A_n}$ is Lipschitz for each $n$. Consequently,
since $\overline{C}\cap A_n\subset\overline{(\mathrm{int}\,C)\cap A_n}\;$,
$F|_{(\mathrm{int}\,C)\cap A_n}$ has
a unique Lipschitz extension $F_n^*\colon \overline{C}\cap A_n\to Y$.
Since, for $n_2>n_1$, $F^*_{n_2}$ obviously extends $F^*_{n_1}$, there
exists a unique continuous $F^*\colon \overline{C}\to Y$ which
extends each $F^*_n$. Since, for each $n$,
any Lipschitz control function for $F|_{(\mathrm{int}\,C)\cap A_n}$
has a Lipschitz extension to $D_n:=\overline{C}\cap A_n$,
Observation~\ref{O:exten}(b) gives that $F^*|_{D_n}$ has a Lipschitz control
function. Moreover,
$\bigcup D_n=\overline{C}$ and
$\mathrm{dist}(D_n,\overline{C}\setminus D_{n+1})=
\mathrm{dist}(D_n,\overline{C}\setminus A_{n+1})\ge
\mathrm{dist}(A_n,X\setminus A_{n+1})>0$ for each $n$. Applying
Lemma~\ref{L:hartman} with $D:=\overline{C}$, we obtain that $F^*$ is
d.c.\ on $\overline{C}$.
\end{proof}

Now we will prove the main result of the present section.
For the definition of modulus of convexity of power type 2
see the text before Proposition \ref{bilinear}.

\begin{theorem}\label{T:rozsir}
Let $X,Y$ be normed linear spaces, $C\subset X$ a convex set with a
nonempty interior and $F\colon C\to Y$  a d.c.\ mapping. Let
$A\supset C$ be an open convex set in $X$.
Suppose that
$X$ admits a renorming with modulus of convexity of
power type 2, and either $C$ is closed or $Y$ is a Banach space.
Then the following assertions are equivalent.
\begin{enumerate}
\item[(i)] $F$ admits a d.c.\ extension $\widehat{F}\colon A\to Y$.
\item[(ii)] Some control function $f$ of $F$ admits a continuous convex extension
$\widehat{f}\colon A\to \R$.
\item[(iii)] There exists a nondecreasing sequence $\{D_n\}$ of open convex sets
such that $A=\bigcup D_n$ and,
for each $n$,
$(\mathrm{int}\,C)\cap D_n\ne\emptyset$ and
the restriction of $F$ to
$(\mathrm{int}\,C)\cap D_n$ has a Lipschitz control function.
\end{enumerate}
\end{theorem}

\begin{proof}
The implication $(i)\Rightarrow(ii)$ is trivial, while
$(ii)\Rightarrow(iii)$ follows immediately from Lemma~\ref{K} applied to $\hat f$.

Let us prove $(iii)\Rightarrow(i)$.
By translation we can suppose that $0\in (\mathrm{int}\,C)\cap D_1$. The
sets $A_n:=D_n\cap U(0,n)$ ($n\in\N$) form a sequnce of
bounded open convex sets such that
$A_n\sipka A$, $0\in (\mathrm{int}\,C)\cap A_1$ and, for each $n$,
\begin{equation}\label{contr}
\text{
$F|_{(\mathrm{int}\,C)\cap A_n}$ has a Lipschitz control function $f_n$.
}
\end{equation}
First we will extend $F$ to a mapping
$F^*\colon \overline{C}\cap A\to Y$. If $C$ is closed, then
$\overline{C}\cap A=C$; so we put $F^*=F$. If $C$ is not closed,
$Y$ is a Banach space by the assumptions. Proposition~\ref{P:lip} and
\eqref{contr} imply that $F$ is Lipschitz on
$(\mathrm{int}\,C)\cap A_n$.
Note that $\overline{C}\cap A_n\subset\overline{(\mathrm{int}\,C)\cap A_n}\,$;
thus $F|_{(\mathrm{int}\,C)\cap A_n}$ has
a unique Lipschitz extension $F_n^*\colon \overline{C}\cap A_n\to Y$.
Since, for $n_2>n_1$, $F^*_{n_2}$ obviously extends $F^*_{n_1}$, there
exists a unique continuous $F^*\colon \overline{C}\cap A\to Y$ which
extends each $F^*_n$.

In both cases ($C$ closed or not), $f_n$ has a Lipschitz extension to
$B_n:=\overline{C}\cap A_n$. By Observation~\ref{O:exten}(b),
\begin{equation}\label{contr2}
\text{
$F^*|_{B_n}$ is Lipschitz and d.c.\ with a Lipschitz control function.
}
\end{equation}

Denote by $\mu$ the Minkowski functional of $C$,
i.e.
\[
\mu(x)=\inf\{t>0: x\in tC\}.
\]
It is well known that $\mu$ is a Lipschitz convex function on $X$
(recall that $0\in\mathrm{int}\,C$), and
$\mu(x)\le1$ if{f} $x\in \overline{C}$. Consider the ``radial projection'' $P$ onto
$\overline{C}$, given by
\[
P(x)=\begin{cases}
x &\text{if $x\in \overline{C}$;}\\
\frac{x}{\mu(x)} &\text{if $x\in X\setminus \overline{C}$.}
\end{cases}
\]
The function $x\mapsto\max\{1,\mu(x)\}$ is convex and Lipschitz, and its
values belong to $[1,\infty)$. The function $t\mapsto\frac{1}{t}$ is
convex and Lipschitz on $[1,\infty)$; consequently, by
Proposition~\ref{P:veza}, the composed function $x\mapsto\frac{1}{\max\{1,\mu(x)\}}$
is d.c.\ on $X$. Moreover, the mapping $B\colon\R\times X\to X$, given
by
\[
B(t,x)=tx\,,
\]
is a continuous bilinear mapping. Since
$P(x)=B\left(\frac{1}{\max\{1,\mu(x)\}},x\right)$, $x\in X$,
 Proposition~\ref{bilinear} implies
that $P$ is d.c.\ on $X$.

Let us show that $\hat{F}:=F^*\circ(P|_A)$ is a d.c.\ extension of $F$
to $A$. The fact that $\hat{F}$ extends $F$ is obvious. To prove that
$\hat{F}$ is d.c., it is sufficient to apply Lemma~\ref{P:main} with
$B:=\overline{C}\cap A$, $\Phi:=P|_A$, $\Psi:=F^*$ (and $A,A_n,B_n$ as above).
Indeed, the assumptions of that lemma are satisfied since
$\Phi(A_n)=P(A_n)\subset \overline{C}\cap A_n=B_n$
(note that $P(A_n)\subset A_n$ because $0\in A_1$) and \eqref{contr2}
holds.
\end{proof}

\begin{remark}
\begin{enumerate}
\item[(a)]
We do not know whether the renorming assumption on $X$ in
Theorem~\ref{T:rozsir} can be omitted or essentially weakened.
\item[(b)]
The condition (ii) in Theorem~\ref{T:rozsir} can be substituted by the
following formally weaker condition:
\begin{enumerate}
\item[(ii')] {\em some control function of $F$ can be extended to a
d.c.\ function on $A$.}
\end{enumerate}
Indeed, if $f_1$ and $f_2$ are continuous convex functions on $A$ such
that $f_1-f_2$ controls $F$ on $C$, then also the
sum $f_1 +f_2$ controls $F$ on $C$.
\end{enumerate}
\end{remark}

\begin{corollary}\label{C:rozsir}
Let $X,Y$ be normed linear spaces, $C\subset X$ be a convex set with a
nonempty interior, and $F\colon C\to Y$ be a
d.c.\ mapping.
Supose that,
the restriction of $F$ to each bounded open convex subset of $C$ has a
Lipschitz control function.
\begin{enumerate}
\item[(a)] If $Y$ is a Banach space, then
           $F$ admits a d.c.\ extension to $\overline{C}$.
\item[(b)] If $X$ admits a renorming with modulus of convexity of
     power type 2, and either $C$ is closed or $Y$ is a Banach space, then
     $F$ admits a d.c.\ extension to the whole $X$.
\end{enumerate}
\end{corollary}

\begin{proof}
Consider the sets $A_n:=U(0,n)$ ($n\in\N$) and apply
Proposition~\ref{L:uzaver} to get (a), and Theorem~\ref{T:rozsir} to get
(b).
\end{proof}

\begin{corollary}\label{elpecka}
Let $X$ be a (subspace of some) $L_p(\mu)$ space with $1<p\le2$. Let
$C\subset X$ be a convex set with a nonempty interior.
\begin{enumerate}
\item[(a)] Each continuous convex function on $C$, which is Lipschitz on every
bounded subset of $\mathrm{int}\,C$, admits a d.c.\ extension to the whole $X$.
\item[(b)] Each Banach space-valued ${\cal C}^{1,1}$ mapping on $C$
admits a d.c.\ extension to the whole $X$.
\end{enumerate}
\end{corollary}

\begin{proof}
It is known (see e.g.\ \cite[p.410]{BeLin}) that $X$, in the $L_p$-norm,
has modulus of convexity of power type
2. Therefore, \cite[Proposition~1.11]{VeZa2} easily implies that
each Banach space-valued ${\cal C}^{1,1}$ mapping on any open
convex subset of $X$ is d.c.\ with a control function that is Lipschitz on
bounded sets. Now, both (a) and (b) follow from Corollary~\ref{C:rozsir}(b).
\end{proof}

For extensions from closed finite-dimensional convex subsets, we have
the following simple corollary.
Recall that a {\em finite-dimensional set} (in a vector space) is a set whose linear span is
finite-dimensional.

\begin{corollary}
Let $X,Y$ be normed linear spaces, and
$F\colon C\to Y$ be a d.c.\ mapping, where $C\subset X$ is a
finite-dimensional closed convex
set. Then the following assertions are equivalent:
\begin{enumerate}
\item[(i)] $F$ admits a d.c.\ extension $\widehat{F}\colon X\to Y$;
\item[(ii)] $F$ has a locally Lipschitz control function $f\colon C\to\R$.
\item[(iii)]For each $x \in C$, there exists $r_x >0$ such that the restriction of
$F$ to $C \cap U(x,r_x)$ has a
 Lipschitz control function.
\end{enumerate}
\end{corollary}

\begin{proof}
The implication $(i)\Rightarrow(ii)$ is obvious since each continuous convex
function on $X$ is locally Lipschitz. The implication $(ii)\Rightarrow(iii)$ is trivial.

Let $(iii)$ hold. Suppose that $0\in C$ and denote
$X_0:=\mathrm{span}\,C$ ($=\mathrm{aff}\,C$). Then $C$, being finite-dimensional,
has a nonempty interior in $X_0$.
 Let $B \subset C$ be a bounded convex set which is open in $X_0$. For each $x \in \overline{B} \cap C$ choose $r_x$
  by $(iii)$ and a Lipschitz convex function $\varphi_x$ on $C \cap U(x,r_x)$
which controls $F$ on $C \cap U(x,r_x)$. Since   $\overline{B} \cap C$ is compact, we can choose $x_1, \dots,x_n$ in
 $\overline{B} \cap C$ such that  $\overline{B} \cap C \subset \bigcup_{i=1}^n \, U(x_i, r_{x_i})$.
Extend $\varphi_{x_i}$ to a Lipschitz convex function $\psi_i$ on $X_0$
(cf. Lemma~\ref{F:1}(c)) and put
   $\psi := \sum_{i=1}^n \psi_i$. Then clearly $\psi|_{B}$ is a Lipschitz control function of $F|_B$.

By Corollary~\ref{C:rozsir}(b), there exists a d.c.\ extension
$F_0\colon X_0\to Y$ of $F$. Let $\pi\colon X\to X_0$ be a
continuous linear projection onto $X_0$. Then the mapping
$\widehat{F}:=F_0\circ\pi$ is a d.c.\ extension of $F$ (cf. \cite[Lemma 1.5(b)]{VeZa1}).
Thus (i) holds and the proof is complete.
\end{proof}


\section{Counterexamples}

\begin{example}\label{pasovec}
There exists a continuous convex function $f$ on the strip
$P:= \R \times [-1,0]$ such that
\begin{enumerate}
\item
$f$ has a d.c.\ extension to $\R^2$, and
\item
$f$ has no convex extension to $\R^2$.
\end{enumerate}
\end{example}

\begin{proof}
For  $(x,y) \in P$, we set
$$ f(x,y) := \sup \{a_t(x,y):\ t \in \R\},\  \text{where}\ \ a_t(x,y) := t^2 + 2t(x-t) + t^2 y.$$
Observe that
\begin{equation}\label{supp}
a_t(\cdot,0) \ \text{is the support affine function to the function}\ \ p(x):=x^2\ \text{at}\ t,
\end{equation}
\begin{equation}\label{nez}
a_t(t,y) \geq 0 \ \ \text{for}\ \ y \in [-1,0],\ \text{and}
\end{equation}
\begin{equation}\label{partder}
\frac{\partial a_t}{\partial x} (z) = 2t,\ \ \ \frac{\partial a_t}{\partial y} (z) = t^2\ \ (z\in \R^2).
\end{equation}
Now fix $\tau \in \R$ and consider a $t \in \R$. Then \eqref{supp} implies $a_t(\tau,0) \leq p(\tau) = \tau^2$,
 so \eqref{partder} gives $a_t(\tau,y) \leq \tau^2$ for $y \in [-1,0]$.
Consequently, $f(\tau,0) = a_{\tau}(\tau,0)= \tau^2$ and
  $f(\tau,y) \leq \tau^2 < \infty$
  for $y \in [-1,0]$. Thus $f$ is a finite convex function on $P$.
Moreover $f\ge0$ on $P$.

Note that $a_t(\tau,0) = -t^2 + 2t \tau \leq 0$ whenever $|t| \geq 2 |\tau|$.
If $z=(z_1,z_2)\in(\tau-1,\tau +1) \times [-1,0]$ and
$|t|\ge2(|\tau|+1)$, then $a_t(z_1,0)\le0$ since $|t|\ge2|z_1|$. For
such $z$ we have
$a_t(z)\le0\le f(z)$ because $a_t(z_1,\cdot)$ is nondecreasing by
\eqref{partder}. It follows that
  \begin{equation}\label{supmen}
  f(z) = \sup \{a_t(z):\ |t| \leq 2(|\tau|+1) \} \ \ \text{for}\ \ z \in (\tau-1,\tau +1) \times [-1,0].
  \end{equation}
  Using \eqref{supmen} and \eqref{partder}, we easily obtain that $f$ is locally Lipschitz on $P$; so it is Lipschitz
   on each bounded subset of $P$. Consequently,  (i) follows from
Corollary~\ref{C:rozsir}.

Now, suppose that (ii) is false, that is, there exists a
convex extension $\tilde{f}\colon\R^2\to\R$ of $f$.
Since $\tilde f(\tau,0) = f(\tau,0)=
    \tau^2$, we have  $\frac{\partial \tilde f}{\partial x} (\tau,0) = 2 \tau$ for each $\tau \in \R$.
    Now we will prove that, for each $\tau >0$,
    \begin{equation}\label{smder}
     d^+_{(0,-1)} \tilde f (\tau,0) =  d^+_{(0,-1)} a_{\tau} (\tau,0) = -\tau^2
    \end{equation}
where $d^+_v g(z)$ denotes the one-sided derivative of $g$ at $z$ in the
direction $v$.
     To this end, choose an arbitrary $\ep >0$ and find $0<\delta< \sqrt{\ep}$ such that $|t^2 - \tau^2| < \ep$ whenever
      $|t-\tau|< \delta$. If $|t-\tau| \geq \delta$ and  $y \in (-\delta^2/\tau^2,0]$, then
      \begin{equation*}
\begin{split} a_{\tau}(\tau,y) - a_t(\tau,y)  &= \tau^2 + \tau^2y + t^2 - 2t\tau -t^2y \\
& = (t-\tau)^2 + (\tau^2-t^2)y \geq \delta^2 + \tau^2 y  > 0.
\end{split}
\end{equation*}
     If $|t-\tau|< \delta$ and $y\le0$, then
     $$ a_{\tau}(\tau,y) - a_t(\tau,y)   = (t-\tau)^2 + (\tau^2-t^2)y \geq \ep y.$$
     Therefore,    $\tilde f(\tau,y) \geq a_{\tau}(\tau,y) \geq \tilde f(\tau,y) + \ep y$
whenever $y \in (-\delta^2/\tau^2,0]$.
Since $\ep>0$ was arbitrary,
we easily obtain \eqref{smder}.

Since $\tilde f$ is convex, the function $v \mapsto d^+_v \tilde f (\tau,0)$ is
positively homogenous and subadditive.
 Therefore, for $\tau>0$, we have
 $$ d^+_{(\tau,-3)} \tilde f (\tau,0) \leq  d^+_{(0,-3)} \tilde f (\tau,0) +
  d^+_{(\tau,0)} \tilde f (\tau,0) = -3\tau^2 + 2 \tau^2 = - \tau^2,$$
  and consequently  $d^+_{(-\tau,3)} \tilde f (\tau,0) \geq \tau^2$. By convexity of $\tilde f$,
  $$ \tilde f (0,3) \geq \tilde f(\tau,0) + d^+_{(-\tau,3)} \tilde f (\tau,0) \geq \tau^2 + \tau ^2 = 2\tau^2.$$
Since  $\tau >0$ was arbitrary, $\tilde f (0,3) = \infty$, a contradiction.
\end{proof}

\begin{example}\label{koulovec}
In $X=\ell_2$,
there exist a closed convex set
$C\subset U(0,1)$ with a nonempty interior
and a continuous convex function
$f\colon C\to\R$ such that:
\begin{enumerate}
\item[(a)] $f$ has a continuous convex extension to
$U(0,1)$, in particular, $f$ is locally Lipschitz on
$C$
(even there exists a nondecreasing sequence of open convex sets
$A_n\sipka U(0,1)$ such that $f$ is Lipschitz on each
$A_n\cap C$);
\item[(b)] $f$ has no d.c.\ extension to  $U(0,r)$ whenever $r>1$.
\end{enumerate}
\end{example}

\begin{proof}
Let $e_n$ be the $n$-th vector of the standard basis of $X=\ell_2$.
For $n,k\in\N$ with $n<k$, put
$$\textstyle
z_{n,k}=(1-\frac1n)e_n + h_n(1-\frac1k)e_k
$$
where $h_n>0$ is such that $(1-\frac1n)^2+h_n^2=1$. Note that
$\|z_{n,k}\|^2=(1-h_n^2)+h_n^2(1-h_k^2)=1-h_n^2 h_k^2$. Put
$$\textstyle
C:=\overline{\mathrm{conv}}\left[
\frac12 B_{X}\cup\{z_{n,k}:\; n,k\in\N,\;n<k\}
\right].
$$
Obviously, $C$ is a closed convex set with a nonempty interior and  $C \subset B_X$. We claim
that $C\subset U(0,1)$.

If this is not the case, there exists $x\in C$ with $\|x\|=1$. Thus
$\sup\langle x,C\rangle=\langle x,x\rangle=1.$ On the other hand,
there exists $n_0\in\N$ such
that
$|\langle x,e_n\rangle|<\frac13$ and $h_n<\frac13$ whenever $n>n_0$.
Thus $|\langle x,z_{n,k}\rangle|\le\frac23$ for $k>n>n_0$. There
exists
$k_0>n_0$ such that
$|\langle x,e_k\rangle|<\frac{1}{2n_0}$ whenever $k>k_0$. Hence, for
$n\le n_0$ and $k>k_0$, we have
$|\langle x,z_{n,k}\rangle|\le(1-\frac1n)+\frac{1}{2n_0}\le
1-\frac{1}{n_0}+\frac{1}{2n_0}=1-\frac{1}{2n_0}$.
Since obviously
$\sup\langle x,\frac12 B_{X}\rangle=\frac12$, we obtain
\begin{align*}
\sup\langle x,C\rangle&=\textstyle
\max\bigl\{\frac12 , \sup\left\{\langle x,z_{n,k}\rangle:\; n,k\in\N,\;n<k\right\}\bigr\}\\
&\le\textstyle
\max\bigl[\{\frac23, 1-\frac{1}{2n_0}\} \cup
\{\|z_{n,k}\|:\; n<k\le k_0,\;n\le n_0\}\bigr]<1.
\end{align*}
This contradiction proves our claim.

The function $x\mapsto 1-\|x\|$ is positive, continuous and concave on
$U(0,1)$. Since the function $t\mapsto\frac1t$ is convex
and decreasing on $(0,\infty)$, the composed function
$g(x)=\frac{1}{1-\|x\|}$ is convex
continuous, and hence locally Lipschitz, on $U(0,1)$. Thus $f:=g|_C$
satisfies (a) by Lemma~\ref{K}. Let us show (b). By Lemma~\ref{ndc}, it
suffices to prove
that $g$ is unbounded on subsets of $C$
of arbitrarily
small diameter. Fix $n\in\N$. For any
two distinct indices $k,l>n$,  we have
$$\textstyle
\|z_{n,k}-z_{n,l}\|^2=
h_n^2\left[(1-\frac1k)^2+(1-\frac1l)^2\right]\le 2h_n^2
$$
which implies $\mathrm{diam}\,\{z_{n,k}: k>n\}\le\sqrt{2}\,h_n$.
This completes the proof since $g(z_{n,k})\to\infty$ as $k\to\infty$.
\end{proof}


\section{Extensions of convex functions from subspaces}

Let $Y$ be a closed subspace of a normed linear space $X$, and
$f\colon Y\to \R$ a continuous convex function. The present section
concerns the problem of existence of a continuous convex extension
$\hat{f}\colon X\to\R$ of $f$.

An example of nonexistence of $\hat f$ was given in \cite[Example~4.2]{BMV}.
On the other hand, it is easy and well known that such $\hat f$ exists if either $Y$ is
complemented in $X$ or $f$ is Lipschitz (see, e.g., \cite{BMV}). Borwein
and Vanderwerff proved in \cite[Fact, p.1801]{BV} that $\hat f$ exists
whenever $f$ is bounded on each bounded subset of $Y$; however, this
sufficient condition is not necessary (see Remark~\ref{nenutna}).
The following theorem contains a necessary and sufficient condition $(iv)$
of the same type, but the proof is more difficult and uses different
methods.
Our main new observation is that a modification of Hartman's construction
from \cite{H} gives the implication $(iv)\Rightarrow(ii)$; and we use also
the implication $(ii)\Rightarrow(i)$
already proved in \cite{BMV}.


\begin{theorem}\label{indiv}
Let $X$ be a normed linear space, $Y \subset X$ its closed subspace,
and $f\colon Y \to \R$ a continuous convex function.
 Then the following statements are equivalent.
 \begin{enumerate}
 \item
 The function $f$ admits a continuous convex extension to $X$.
 \item
 There exists a continuous convex function $g\colon X \to \R$ such that $f \leq g|_{Y}$.
 \item
 $f$ admits a d.c.\ extension to $X$.
 \item
 There exists a sequence $\{C_n\}$ of nonempty open convex subsets of $X$ such that $C_n \nearrow X$ and
  $f$ is bounded on each set $C_n \cap Y\ (n \in \N)$.
 \item
 There exists a sequence $\{B_n\}$ of nonempty open convex subsets of $X$ such that $B_n \nearrow X$ and
  $f$ is Lipschitz on each set $B_n \cap Y\ (n \in \N)$.
   \end{enumerate}
\end{theorem}
\begin{proof}
$(i)\Rightarrow(ii)$ is obvious, while $(ii)\Rightarrow(i)$ was proved
in \cite[Lemma~4.7]{BMV} for $X$ a Banach space. However, the proof
therein works also in the normed linear case (note that the convex
extension $\tilde{f}$ from \cite[Lemma~4.7]{BMV} is continuous since it
is locally upper bounded).

$(i)\Rightarrow(iii)$ is trivial.

$(iii)\Rightarrow(ii)$. If (iii) holds, there exist continuous convex
functions $u,v$ on $X$ such that $u(y)-v(y)=f(y)$ for $y\in Y$. Choose a
continuous affine function $a$ on $X$ such that $a\le v$. Then
$$
u(y)-a(y)=f(y)+(v(y)-a(y))\ge f(y),\quad y\in Y;
$$
so we can put $g:=u-a$.

$(i)\Rightarrow(v)$ follows immediately applying Lemma~\ref{K} to a
continuous convex extension $\hat{f}$ of $f$.

$(v)\Rightarrow(iv)$. Clearly, it suffices to put
$C_n=B_n\cap U(0,n)$ ($n\in\N$).

It remains to prove $(iv) \Rightarrow (ii)$.
Using Lemma~\ref{L:dn} and an obvious shift of indices,
it is easy to find a sequence $\{D_n\}$ of bounded open convex
subset of $X$ such that (for each $n$)
$D_n \cap Y \neq\emptyset$, $f$ is bounded on $D_n\cap Y$,
$d_n: = \dist(D_n, X \setminus D_{n+1})>0$, and $D_n \nearrow X$.

 Now we will construct inductively a sequence $(g_n)_{n\in \N}$ of 
functions on $X$ such that, for each $n \in \N$,
 \begin{enumerate}
 \item[(a)]
 $g_n$ is convex and Lipschitz;
 \item[(b)]
 $g_n = g_{n-1}$ \ on\ $D_{n-1}$ whenever  $n>1$;
 \item[(c)]
 $g_n \geq f$ \ on \ $D_{n+1} \cap Y$.
\end{enumerate}
Set   $M_n := \sup\{f(y):\ y \in D_n \cap Y\}$; by the assumptions $M_n < \infty$.

Define $g_1(x) := M_2,\ x \in X$. Then the conditions (a), (b), (c) clearly hold for $n=1$.

Now suppose that $k>1$ and  we already have
$g_1,\dots, g_{k-1}$ such that  (a), (b), (c)  hold for each $1 \leq n <k$.
 We can clearly choose $a \in \R$ such that $g_{k-1}(x) \geq a$ for each $x \in D_{k-1}$, and then $b > 0$ such that
  $a + b\, d_{k-1} \geq M_{k+1}$. Define
  $$ g_k(x) := \max \{g_{k-1}(x), a + b\, \dist (x, D_{k-1})\},\ \ \ x \in X.$$
  We will show that  the conditions (a), (b), (c) hold for $n=k$.
   The validity of (a) is obvious. If $x \in D_{k-1}$, then $g_k(x) = \max\{g_{k-1}(x),a\} = g_{k-1}(x)$; so (b) holds.

   Now consider an arbitrary  $y \in D_{k+1} \cap Y$. If $y \in D_{k-1}$, using (b) for $n=k$ and (c) for $n=k-1$,
we obtain $g_k(y) = g_{k-1}(y) \geq f(y)$. If $y \in D_k \setminus D_{k-1}$, using the definition of $g_k$ and (c) for
    $n=k-1$, we also obtain $g_k(y) \ge g_{k-1}(y) \geq f(y)$. If $y \in D_{k+1} \setminus D_k$, then
    $$ g_k(y) \geq a + b \, \dist(y, D_{k-1}) \geq a + b\, d_{k-1} \geq M_{k+1} \geq f(y).$$
Now, for each $x \in X$, the sequence $\{g_n(x)\}$ is constant for large $n$'s, hence  $g(x):= \lim_{n\to \infty} g_n(x)$
 is defined on $X$. Since $g = g_n$ on $D_n$ by (b), 
the conditions (a) and (c) easily imply that $g$ is a continuous
  convex function on $X$ such that $f \leq g|_{Y}$.
\end{proof}

\begin{remark}\label{nenutna}
As already mentioned, (i) holds whenever
\begin{enumerate}
\item[$(*)$] $f$ is bounded on each bounded subset of $Y$
\end{enumerate}
(indeed, (iv) holds with $C_n:=U(0,n)$). To see that $(*)$ is not
necessary with $Y\ne X$, consider an arbitrary infinite dimensional
Banach space $X$, a closed subspace $Y$ of finite codimension in $X$,
and a continuous convex function $f$ on $Y$ which is unbounded on some
bounded set (for its existence, see \cite{BFV}).
\end{remark}

 \begin{theorem}\label{univ}
Let $X$ be a normed linear space and $Y \subset X$ its closed subspace.
Then the following statements are equivalent.
\begin{enumerate}
\item
Each continuous convex function $f\colon Y\to\R$ admits a continuous convex extension to $X$.
\item
If $\{C_n\}$ is a sequence of open convex subsets of $Y$ such that
$C_n\sipka Y$, then there exists
a sequence $\{D_n\}$ of open convex subsets of $X$ such that
$D_n\sipka X$ and $D_n\cap Y\subset C_n$.
\item
If $\{C_n\}$ is a sequence of open convex subsets of $Y$ such that
$C_n\sipka Y$, then there exists
a sequence $\{\tilde{C}_n\}$ of open convex subsets of $X$ such that
$\tilde{C}_n\sipka X$ and $\tilde{C}_n\cap Y = C_n$.
\end{enumerate}
\end{theorem}

\begin{proof}
$(i)\Rightarrow(ii)$.
Let $\{C_n\}$ be as in (ii). Using Lemma \ref{L:dn}, we can (and do) suppose
that $C_n\ne\emptyset$ and $C_n\cc C_{n+1}$ in $Y$ ($n\in\N$). Fix $a\in C_1$
and put $C_0:=\{a\}$. Choose $\epsilon_n>0$ such that
$C_n+\ep_n B_Y\subset C_{n+1}$ ($n\ge0$), and consider the function
\[
f(y):=\sum_{n=0}^\infty \frac{1}{\ep_n}\,\mathrm{dist}(y,C_n)\,,\quad
y\in Y.
\]
It is easy to see that $f$ is a continuous convex function on $Y$;
therefore it admits a continuous convex extension $\hat{f}$ to $X$ by
(i). Let us show that the sets $D_n:=\{x\in X: \hat{f}(x)<n\}$ ($n\in\N$) have the
desired properties. Obviously, they are convex and open, and $D_n\sipka
X$. Consider $n\in\N$ and $y\in Y\setminus C_n$. Since
$\mathrm{dist}(y,C_k)\ge \epsilon_k$ for $0\le k<n$, we have
$$
f(y)\ge\sum_{k=0}^{n-1}\frac{1}{\epsilon_k}\,\mathrm{dist}(y,C_k)\,\ge n.
$$
This shows that $D_n\cap Y\subset C_n$.

$(ii)\Rightarrow(i)$.
Let $f$ be as in (i). Then the sets
$C_n:=\{y\in Y: f(y)<n,\;\|y\|<n\}$ ($n\in\N$) are open convex and satisfy
$C_n\sipka Y$. Observe that $f$ is bounded on each $C_n$ by
Lemma~\ref{F:1}(a).
Find $D_n$ ($n\in\N$) by (ii). Since $D_n\cap Y\subset C_n$, the sequence
$\{D_n\}$ satisfies the condition (iv) of Theorem~\ref{indiv}, and so (i)
follows.

$(ii)\Rightarrow(iii)$.
Let $\{C_n\}$ be as in (iii). Find $D_n$ ($n\in\N$) by (ii). Choose
$n_0\in\N$ such that $D_{n_0}\cap Y\ne\emptyset$. For $n\ge n_0$, put
$\tilde{C}_n:=\mathrm{conv}(D_n\cup C_n)$. By Lemma~\ref{conv}(a),(c),
we have that $\tilde{C}_n\cap Y=C_n$ and the convex set $\tilde{C}_n$ is open
for any $n\ge n_0$. Let $n_1$ be the smallest index such that
$C_{n_1}\ne\emptyset$.
Fix $c\in C_{n_1}$ and choose $r>0$ such that
$U(c,r)\subset\tilde{C}_{n_0}$ and
$U(c,r)\cap Y\subset C_{n_1}$. Put $\tilde{C}_n=\emptyset$ for
$1\le n<n_1$, and $\tilde{C}_n:=\mathrm{conv}(U(c,r)\cup C_n)$ for
$n_1\le n<n_0$. Using Lemma~\ref{conv}(a),(c) as above, we easily obtain
that the sequence $\{\tilde{C}_n\}$ has the desired properties.
The reverse implication $(iii)\Rightarrow(ii)$ is obvious.
\end{proof}

As an application of Theorem~\ref{univ}, we give an alternative proof
(see Theorem~\ref{subspace}) of the fact that separability of the quotient
space $X/Y$ is sufficient for extendability of all continuous convex
functions on $Y$ to $X$. This was proved in \cite{BMV} for Banach spaces
using a condition about nets in $Y^*$, equivalent to
(i) of Theorem \ref{univ}, together with  Rosenthal's extension theorem.
Our proof (for general normed
linear spaces) is based on Theorem~\ref{univ} and on the following
elementary lemma.

\begin{lemma}\label{kuzeliky}
Let $Y$ be a closed subspace of a normed linear space $X$. Let $B=r B_X$ for some
$r>0$. Then, for any $x\in X$, there exists $y_x\in Y$ such that
\begin{equation}\label{E:kuzeliky}
\mathrm{conv}[(x+B)\cup B] \cap Y
\subset \mathrm{conv}[\{y_x\}\cup 8B].
\end{equation}
\end{lemma}

\begin{proof}
If $x=0$, $y_x=0$ works. For
$x\ne0$,
denote $P:=\mathrm{conv}[(x+B)\cup B]\cap Y$ ($\ne\emptyset$) and
$s:=\sup\{\|y\|:y\in P\}$, and
choose $u_0\in P$ such that $\|u_0\|>s-r$.
Observe that $u_0\ne0$ since $s\ge\sup\{\|y\|:y\in B\cap Y\}=r$. We claim
that $y_x=8u_0$ works.

Fix $a_0\in x+B$, $b_0\in B$ and $\lambda\in[0,1]$ such that
$u_0=(1-\lambda)a_0+\lambda b_0$.
It suffices to prove that $P\setminus B\subset \mathrm{conv}[\{y_x\}\cup 8B].$
Given $u\in P\setminus B$, choose $a\in x+B$ and $b\in B$ such that
$u\in[a,b]$. Note that $a\ne b$ since $u\notin B$.
Put $v=(1-\lambda)a+\lambda b$ and observe that $\|v-u_0\|\le 2r$.
Consider the half-line
$H:=\{v+t(a-b): t\ge0\}$. Let $v_1\in H$ be such that $\|v_1-v\|=5r$.
Then $v_1\in u_0+7B$ since
$\|v_1-u_0\|\le\|v_1-v\|+\|v-u_0\|\le 7r$.

We claim that no point $y\in H$ with $\|y-v\|>5r$ can belong to $P$ since it satisfies
$\|y\|>s$.
Indeed, since $v\in[y,b]$,
\begin{align*}
\|y\| &\ge\|y-b\|-\|b\|=\|y-v\|+\|v-b\|-\|b\|
\ge \|y-v\| +\|v\|-2\|b\| \\
& \ge \|y-v\| +\|u_0\|-\|u_0-v\|-2\|b\| > 5r+(s-r)-2r-2r=s.
\end{align*}
Consequently, if $u\in H$ then $u\in[v_1,v]$,
and if $u\in[a,b]\setminus H$ then $u\in[v,b]$. In both cases,
$u\in[v_1,b]\subset\mathrm{conv}[(u_0+7B)\cup B]$.
To finish, observe that $u_0+7B=\frac18(8u_0)+\frac78(8B)$ implies
$$
u\in \mathrm{conv}[(u_0+7B)\cup B]\subset
\mathrm{conv}[\{8u_0\}\cup 8B].
$$
\end{proof}

\begin{theorem}[{\cite[Corollary~4.10]{BMV}} for $X$ Banach]\label{subspace}
Let $Y$ be a closed subspace of a normed linear space $X$ such
that $X/Y$ is separable. Then each continuous convex function
$f\colon Y\to\R$ admits a continous convex extension to $X$.
\end{theorem}

\begin{proof}
It suffices to verify the condition (ii) of
Theorem~\ref{univ}.
Let $C_1\subset C_2\subset\ldots$ be
open convex subsets of $Y$ such that
$\bigcup_n C_n=Y$.
We can (and do) suppose that
$0\in\mathrm{int}_Y\,C_1$. Fix $r>0$ such that
\begin{equation}\label{E:8r}
8 r B_Y\subset C_1.
\end{equation}
Fix a dense sequence $\{\xi_n\}_{n\in\N}\subset X/Y$ and,
for each $n$, choose an arbitrary $z_n\in\xi_n$. The sets
$Z_n:=\mathrm{conv}\{z_1,\ldots,z_n\}$ ($n\in\N$) form a nondecreasing
sequence of compact convex sets such that the union
$\bigcup_n(Z_n+Y)$ is dense in $X$. Define $Z_0=\emptyset$.

{\it Claim.} There exists an increasing sequence of
integers $\{k_n\}_{n\ge0}$ such that $k_0=1$ and, for each $n$,
\begin{equation}\label{E:claim}
\mathrm{conv}(Z_n\cup B)\cap Y\subset C_{k_n}\qquad
\text{where $B=rB_X$.}
\end{equation}
To prove this, we shall proceed by induction with respect to $n$.
Observe that \eqref{E:claim} is satisfied for $n=0$ and $k_0=1$.
Suppose we already have
$k_0,\ldots,k_{n-1}$. Since $Z_n$ is compact, there exists a finite set
$F\subset Z_n$ such that $Z_n\subset F+B$. For any $x\in F$, fix $y_x\in Y$
satisfying \eqref{E:kuzeliky}. Choose an integer $k_n>k_{n-1}$ such that
$y_x\in C_{k_n}$ for each $x\in F$. Then, using \eqref{Val}, we obtain
\[
\mathrm{conv}(Z_n\cup B)=
\bigcup_{z\in Z_n}\mathrm{conv}(\{z\}\cup B)\subset
\bigcup_{x\in F}\mathrm{conv}((x+B)\cup B).
\]
Consequently, using \eqref{E:kuzeliky} and Lemma \ref{conv}(a), we obtain
\begin{align*}
\mathrm{conv}(Z_n\cup B)\cap Y &\subset
\bigcup_{x\in F}[\mathrm{conv}((x+B)\cup B)\cap Y] \\
&\subset
\bigcup_{x\in F}[\mathrm{conv}(\{y_x\}\cup 8B)\cap Y] \\
&=
\bigcup_{x\in F} \mathrm{conv}[(8B\cap Y)\cup \{y_x\}] \subset C_{k_n}
\end{align*}
since, by \eqref{E:8r}, $(8B\cap Y)\cup \{y_x\} \subset C_{k_n}$ for each $x\in F$.
This proves our Claim.

For each $j\in\N$, let $n(j)$ be the unique nonnegative integer
with $k_{n(j)}\le j <k_{n(j)+1}$.
Let us define a nondecreasing sequence $\{D_j\}_{j\in\N}$ of
open convex sets by
$$
D_j:=\mathrm{int}[\mathrm{conv}(Z_{n(j)}\cup B\cup C_j)].
$$
By Lemma~\ref{conv}(a) and \eqref{E:claim}, we have
\begin{align*}
Y\cap D_j&\subset
Y\cap \mathrm{conv}(Z_{n(j)}\cup B\cup C_j)\\
&=
Y\cap\mathrm{conv}[\mathrm{conv}(Z_{n(j)}\cup B)\cup C_j]
\\&=
\mathrm{conv}\bigl\{[Y\cap\mathrm{conv}(Z_{n(j)}\cup B)]\cup C_j\bigr\}
\\&\subset
\mathrm{conv}\{C_{k_{n(j)}}\cup C_j\}=C_j.
\end{align*}
It remains to prove that $\bigcup_j D_j=X$. By Lemma~\ref{conv}(b), this
is equivalent to say
that $\bigcup_j D_j$ is dense in $X$. Since, for each $j$, $D_j$ is
dense in $\tilde{D}_j:=\mathrm{conv}(Z_{n(j)}\cup B\cup C_j)$, it
suffices to show that $\bigcup_j \tilde{D}_j$ is dense.
Note that $H:=\frac12\,\bigcup_n(Z_n+Y)$ is dense in $X$ since
$\bigcup_n(Z_n+Y)$ is dense.
If $h\in H$ then
$h=\frac12(z+y)$ with $z\in Z_n$ for some $n$, and $y\in Y$. Then, for
sufficiently large $j$, we have $z\in Z_{n(j)}$ and $y\in C_j$, and
hence $h\in \tilde{D}_j$. Consequently,  $\bigcup_j \tilde{D}_j$ 
is dense since it
contains $H$.
\end{proof}


\bigskip
\bigskip

\subsection*{Acknowledgement}
The research of the first author was partially supported by the Ministero
dell'Istruzione, dell'Universit\`a e della Ricerca
 of Italy.
The research of the second author was partially supported by the grant
GA\v CR 201/06/0198  from the Grant Agency of
Czech Republic and partially supported by the grant MSM  0021620839 from
 the Czech Ministry of Education.

\end{document}